\documentclass[reqno,10pt,centertags,draft]{amsart}
\usepackage{amsmath,amsthm,amscd,amssymb,latexsym,upref}
\date{\today}

\newcommand{\bbN}{{\mathbb{N}}}
\newcommand{\bbR}{{\mathbb{R}}}

\newcommand{\bbC}{{\mathbb{C}}}

\newcommand{\cB}{{\mathcal B}}

\newcommand{\cH}{{\mathcal H}}

\newcommand{\cS}{{\mathcal S}}

\newcommand{\cX}{{\mathcal X}}

\newcommand{\dott}{\,\cdot\,}
\newcommand{\no}{\notag}
\newcommand{\lb}{\label}
\newcommand{\f}{\frac}

\newcommand{\ol}{\overline}

\newcommand{\wti}{\widetilde}

\newcommand{\dom}{\text{\rm{dom}}}

\newcommand{\bi}{\bibitem}
\newcommand{\hatt}{\widehat}
\newcommand{\beq}{\begin{equation}}
\newcommand{\eeq}{\end{equation}}
\newcommand{\ba}{\begin{align}}
\newcommand{\ea}{\end{align}}

\newcommand{\tr}{\text{\rm{tr}}}

\renewcommand{\Im}{\text{\rm Im}}
\renewcommand{\ln}{\text{\rm ln}}

\newcommand{\Om}{\Omega}
\newcommand{\dOm}{{\partial\Omega}}
\newcommand{\si}{\sigma}

\newcommand{\ga}{\gamma}

\newcommand{\eps}{\varepsilon}

\newcommand{\LOm}{L^2(\Om;d^nx)}
\newcommand{\LdOm}{L^2(\dOm;d^{n-1}\si)}

\allowdisplaybreaks \numberwithin{equation}{section}

\newtheorem{theorem}{Theorem}[section]

\newtheorem{lemma}[theorem]{Lemma}

\newtheorem{hypothesis}[theorem]{Hypothesis}
\theoremstyle{definition}

\newtheorem{remark}[theorem]{Remark}

\begin{document}

\title[Dirichlet-to-Neumann Maps and Infinite Determinants]{On 
Dirichlet-to-Neumann Maps, Nonlocal Interactions, and Some Applications to \\ 
Fredholm  Determinants}
\author[F.\ Gesztesy, M.\ Mitrea, and M.\ Zinchenko]{Fritz Gesztesy, Marius Mitrea, and Maxim Zinchenko}
\address{Department of Mathematics,
University of Missouri, Columbia, MO 65211, USA}
\email{gesztesyf@missouri.edu}
\urladdr{http://www.math.missouri.edu/personnel/faculty/gesztesyf.html}
\address{Department of Mathematics, University of
Missouri, Columbia, MO 65211, USA}
\email{mitream@missouri.edu}
\urladdr{http://www.math.missouri.edu/personnel/faculty/mitream.html}
\address{Department of Mathematics,
California Institute of Technology, Pasadena, CA 91125, USA}
\email{maxim@caltech.edu}
\urladdr{http://math.caltech.edu/$\sim$maxim}
\dedicatory{Dedicated with great pleasure to Willi Plessas on the occasion of his 60th birthday}
\thanks{Based upon work partially supported by the US National Science
Foundation under Grant Nos.\ DMS-0400639 and FRG-0456306.}
\thanks{{\it Few Body Systems} {\bf 47}, 49--64 (2010).}
\date{\today}
\subjclass[2000]{Primary: 47B10, 47G10, Secondary: 34B27, 34L40.}
\keywords{Fredholm determinants, non-self-adjoint operators, multi-dimensional 
Schr\"odinger operators, Dirichlet-to-Neumann maps, nonlocal interactions.}

\begin{abstract}
We consider Dirichlet-to-Neumann maps associated with (not necessarily self-adjoint) 
Schr\"odinger operators describing nonlocal interactions in $L^2(\Omega; d^n x)$, 
where $\Om\subset\bbR^n$, $n\in\bbN$, $n\geq 2$, are open sets with a compact, nonempty boundary $\partial\Om$ satisfying certain regularity conditions. As an application we describe a reduction of a certain ratio of Fredholm perturbation determinants associated with operators in 
$L^2(\Om; d^n x)$ to Fredholm perturbation determinants associated with operators in $L^2(\partial\Om; d^{n-1}\sigma)$, $n\in\bbN$, $n\geq 2$. This leads to an extension of a variant of a celebrated formula due to Jost and Pais, which reduces the Fredholm perturbation determinant associated with a Schr\"odinger operator on the half-line $(0,\infty)$, in the case of local interactions, to a simple Wronski determinant of appropriate distributional  solutions of the underlying Schr\"odinger equation. 
\end{abstract}

\maketitle

\section{Introduction}\label{s1}

Since a considerable part of W.\ Plessas' research focuses on various aspects of nonlocal (in particular, separable) interactions, we thought it would be appropriate to derive some of our recent results on Dirichlet-to-Neumann maps and Fredholm determinants in \cite{GMZ07} in the context of nonlocal interactions. 

To illustrate the principle ideas underlying this paper, we briefly recall a
celebrated result of Jost and Pais \cite{JP51}, who proved  in 1951 a spectacular
reduction of the Fredholm determinant associated with the Birman--Schwinger kernel of a one-dimensional Schr\"odinger operator on a half-line, to a  simple Wronski determinant of distributional solutions of the underlying
Schr\"odinger equation. This Wronski determinant also equals the so-called Jost function of the corresponding half-line Schr\"odinger operator. In this paper we prove a certain multi-dimensional variant of this result in the presence of nonlocal (in fact, trace class) interactions.

To describe the result due to Jost and Pais \cite{JP51}, we need a few preparations (we refer to our list of notations at the end of the introduction).
Denoting by $H_{0,+}^D$ and $H_{0,+}^N$ the one-dimensional Dirichlet and Neumann Laplacians  in $L^2((0,\infty);dx)$, and assuming
\begin{equation}
\hatt V\in L^1((0,\infty);dx),   \lb{1.1}
\end{equation}
we introduce the perturbed Schr\"odinger operators
$\hatt H_{+}^D$ and $\hatt H_{+}^N$ in $L^2((0,\infty);dx)$ by
\begin{align}
&\hatt H_{+}^Df=-f''+ \hatt Vf,  \no \\
&f\in \dom\big(\hatt H_{+}^D\big)= \big\{g\in L^2((0,\infty); dx) \,\big|\, g,g'
\in AC([0,R])
\text{ for all $R>0$}, \\
& \hspace*{4.6cm} g(0)=0, \, \big(-g''+ \hatt Vg\big)\in L^2((0,\infty); dx)\big\}, \no \\
&\hatt H_{+}^Nf=-f''+ \hatt Vf,  \no \\
&f\in \dom\big(\hatt H_{+}^N\big)= \big\{g\in L^2((0,\infty); dx) \,\big|\, g,g'
\in AC([0,R])
\text{ for all $R>0$}, \\
& \hspace*{4.5cm} g'(0)=0, \, \big(-g''+ \hatt Vg\big)\in L^2((0,\infty); dx)\big\}. \no
\end{align}
(Here $AC([0,R])$ denotes the set of absolutely continuous functions on $[0,R]$.)
Thus, $\hatt H_{+}^D$ and $\hatt H_{+}^N$ are self-adjoint if and only if $\hatt V$ is
real-valued, but the latter restriction plays no special role in our present context.

A fundamental system of solutions $\phi_+^D(z,\cdot)$,
$\theta_+^D(z,\cdot)$, and the Jost solution $f_+(z,\cdot)$ of
\begin{equation}
-\psi''(z,x)+ \hatt V\psi(z,x)=z\psi(z,x), \quad z\in\bbC\backslash\{0\}, \;
x\geq 0,   \lb{1.4}
\end{equation}
are then introduced via the standard Volterra integral equations
\begin{align}
\phi_+^D(z,x)&=z^{-1/2}\sin(z^{1/2}x)+\int_0^x dx' \, z^{-1/2}\sin(z^{1/2}(x-x'))
\hatt V(x')\phi_+^D(z,x'), \\
\theta_+^D(z,x)&=\cos(z^{1/2}x)+\int_0^x dx' \, z^{-1/2}\sin(z^{1/2}(x-x'))
\hatt V(x')\theta_+^D (z,x'), \\
f_+(z,x)&=e^{iz^{1/2}x}-\int_x^\infty dx' \,
z^{-1/2}\sin(z^{1/2}(x-x')) \hatt V(x')f_+(z,x'),  \lb{1.7} \\
&\hspace*{3.85cm}  z\in\bbC\backslash\{0\}, \; \Im(z^{1/2})\geq 0, \;
x\geq 0.  \no
\end{align}

In addition, we introduce
\begin{equation}
\hatt u=\exp\big(i\arg\big(\hatt V\big)\big)\big|\hatt V\big|^{1/2}, 
\quad \hatt v=\big|\hatt V\big|^{1/2}, \, \text{ so
that } \, \hatt V= \hatt u\, \hatt v,
\end{equation}
and denote by $I_+$ the identity operator in $L^2((0,\infty); dx)$. Moreover,
we denote by
\begin{equation}
W(f,g)(x)=f(x)g'(x)-f'(x)g(x), \quad x \geq 0,
\end{equation}
the Wronskian of $f$ and $g$, where $f,g \in C^1([0,\infty))$.

Then, the following results hold:

\begin{theorem} \lb{t1.1}
Assume $\hatt V\in L^1((0,\infty);dx)$ and let $z\in\bbC\backslash [0,\infty)$
with $\Im(z^{1/2})>0$. Then,
\begin{equation}
\ol{\hatt u\big(H_{0,+}^D-z I_+\big)^{-1} \hatt v}, \,
\ol{\hatt u\big(H_{0,+}^N-z I_+\big)^{-1} \hatt v} \in \cB_1\big(L^2((0,\infty);dx)\big) 
\lb{1.10} 
\end{equation}
and
\begin{align}
\det\Big(I_+ +\ol{\hatt u\big(H_{0, +}^D-z I_+\big)^{-1} \hatt v}\,\Big) &=
1+z^{-1/2}\int_0^\infty dx\, \sin(z^{1/2}x) \hatt V(x)f_+(z,x)   \no \\
&= W(f_+(z,\cdot),\phi_+^D(z,\cdot)) = f_+(z,0),    \lb{1.11}  \\
\det\Big(I_+ +\ol{\hatt u\big(H_{0, +}^N-z I_+\big)^{-1} \hatt v}\,\Big)
&= 1+ i z^{-1/2} \int_0^\infty
dx\, \cos(z^{1/2}x) \hatt V(x)f_+(z,x) \no  \\
&= - \frac{W(f_+(z,\cdot),\theta_+^D (z,\cdot))}{i z^{1/2}} =
\frac{f_+'(z,0)}{i z^{1/2}}.    \lb{1.12}
\end{align}
\end{theorem}

Equation \eqref{1.11} is the modern formulation of the classical result due to Jost and
Pais \cite{JP51} (cf.\ also \cite{BF60} and the detailed discussion in \cite{GM03}). Performing calculations similar to Section 4 in \cite{GM03} for the pair of 
operators $H_{0,+}^N$ and $\hatt H_+^N$, one obtains the analogous result 
\eqref{1.12}.  

For an extension of the classical Jost--Pais formula \eqref{1.11} from local 
interactions $\hatt V$ to nonlocal interactions we refer to \cite{WB71} (see 
also \cite{SW71}).

We emphasize that \eqref{1.11} and \eqref{1.12} exhibit the
remarkable fact that the Fredholm determinant associated with trace
class operators in the infinite-dimensional space $L^2((0,\infty);
dx)$ is reduced to a simple Wronski determinant of $\bbC$-valued
distributional solutions of \eqref{1.4}. This fact goes back to Jost
and Pais \cite{JP51} (see also \cite{BF60}, \cite{GM03}, \cite{Ne72}, \cite{Ne80},
\cite[Sect.\ 12.1.2]{Ne02}, \cite{Si00}, \cite[Proposition 5.7]{Si05},
and the extensive literature cited in these references).
Next, we explore the extent to which
this fact may generalize to higher dimensions $n\in\bbN$, $n\geq 2$.
While a straightforward generalization of \eqref{1.11}, \eqref{1.12}
appears to be difficult, we will next derive a formula for the ratio
of such determinants which indeed permits a direct extension to
higher dimensions.

For this purpose we introduce the boundary trace operators
$\ga_D$ (Dirichlet trace) and $\ga_N$ (Neumann trace) which, in the
current one-dimensional half-line situation, are just the functionals,
\begin{equation}
\ga_D \colon \begin{cases} C([0,\infty)) \to \bbC, \\
\hspace*{1.3cm} g \mapsto g(0), \end{cases}   \quad
\ga_N \colon \begin{cases}C^1([0,\infty)) \to \bbC,  \\
\hspace*{1.43cm} h \mapsto   - h'(0). \end{cases}
\end{equation}
In addition, we denote by $m_{0,+}^D$, $m_+^D$, $m_{0,+}^N$, and $m_+^N$
the Weyl--Titchmarsh $m$-functions corresponding to $H_{0,+}^D$,
$\hatt H_{+}^D$, $H_{0,+}^N$, and $\hatt H_{+}^N$, respectively, that is,
\begin{align}
m_{0,+}^D(z) &= i z^{1/2}, \qquad m_{0,+}^N (z)= -\frac{1}{m_{0,+}^D(z)}
= i z^{-1/2},  \lb{1.14} \\
m_{+}^D(z) &= \frac{f_+'(z,0)}{f_+(z,0)}, \quad m_{+}^N (z)=
-\frac{1}{m_{+}^D(z)} = -\frac{f_+(z,0)}{f_+'(z,0)}.  \lb{1.15}
\end{align}

Then we obtain the following result for the ratio of the perturbation determinants in
\eqref{1.11} and \eqref{1.12}:

\begin{theorem} \lb{t1.2}
Assume $\hatt V\in L^1((0,\infty);dx)$ and let $z\in\bbC\backslash
\si\big(\hatt H_+^D\big)$ with $\Im(z^{1/2})>0$. Then,
\begin{align}
& \frac{\det\Big(I_+ +\ol{\hatt u\big(H_{0, +}^N-z I_+\big)^{-1} \hatt v}\,\Big)}
{\det\Big(I_+ +\ol{\hatt u\big(H_{0, +}^D-z I_+\big)^{-1} \hatt v}\,\Big)}
 = 1 - \Big(\,\ol{\ga_N\big(\hatt H_+^D-z I_+\big)^{-1} \hatt V
\big[\ga_D(H_{0,+}^N-\ol{z}I_+)^{-1}\big]^*}\,\Big)   \lb{1.16}  \\
& \quad = \f{W(f_+(z),\phi_+^N(z))}{i z^{1/2}W(f_+(z),\phi_+^D(z))} =
\f{f'_+(z,0)}{i z^{1/2}f_+(z,0)} = \f{m_+^D(z)}{m_{0,+}^D(z)} =
\f{m_{0,+}^N(z)}{m_+^N(z)}.   \lb{1.17}
\end{align}
\end{theorem}

The multi-dimensional generalizations to Schr\"odinger operators in $L^2(\Om;d^n x)$, corresponding to an open set $\Om \subset \bbR^n$ with compact, nonempty boundary $\partial\Om$, more precisely, the proper operator-valued generalization of the Weyl--Titchmarsh function $m_+^D(z)$ is then given by the Dirichlet-to-Neumann map, denoted by $M_{\Om}^D(z)$ in $L^2(\dOm; d\sigma^{n-1})$. This operator-valued map  indeed played a fundamental role in our extension of \eqref{1.17} to the higher-dimensional case in \cite{GMZ07}.  

We recall the assumptions on the set $\Omega$ in \cite{GMZ07}:

\begin{hypothesis} \lb{h1.3a}
Let $n\in\bbN$, $n\geq 2$, and assume that $\Omega\subset{\bbR}^n$ is
an open set with a compact, nonempty boundary
$\partial\Omega$. In addition,
we assume that one of the following three conditions holds: \\
$(i)$ \, $\Omega$ is of class $C^{1,r}$ for some $1/2 < r <1$; \\
$(ii)$  \hspace*{.0001pt} $\Omega$ is convex; \\
$(iii)$ $\Omega$ is a Lipschitz domain satisfying a {\it uniform exterior ball condition} $($UEBC\,$)$.  
\end{hypothesis}

We note that while $\dOm$ is assumed to be compact, $\Om$ may be unbounded in connection with conditions $(i)$ or $(iii)$. For more details in the context of the notation used in Hypothesis \ref{h1.3a} we refer to \cite[App.\ A]{GMZ07}.

Given the self-adjoint and nonnegative Dirichlet and
Neumann Laplacians $H_{0,\Om}^D$ and $H_{0,\Om}^N$ associated
with the domain $\Om$ in $L^2(\Om; d^n x)$ as defined in \eqref{2.39} and 
\eqref{2.20}, respectively (although, the latter can be described in additional detail 
under the stronger Hypotheses \ref{h1.3a} as compared to Hypothesis \ref{h2.1}, 
cf.\ \cite{GMZ07}), we now introduce 
$\hatt H_{\Om}^D$ and $\hatt H_{\Om}^N$, the Dirichlet and Neumann 
Schr\"odinger operators in $L^2(\Om; d^n x)$ associated with the (local) 
differential expressions $-\Delta + \hatt V(x)$ and Dirichlet and Neumann 
boundary conditions on $\partial\Om$ as follows:
\begin{align}
& \big(\hatt H^D_{\Om}-zI_\Om\big)^{-1} = \big(H^D_{0,\Om}-zI_\Om\big)^{-1}  \no \\
& \quad - \ol{\big(H^D_{0,\Om}-zI_\Om\big)^{-1} \hatt v}
\Big[I_\Om+\ol{\hatt u\big(H^D_{0,\Om}-zI_\Om\big)^{-1} \hatt v}\,\Big]^{-1}
\hatt u\big(H^D_{0,\Om}-zI_\Om\big)^{-1},    \lb{1.17a} \\ 
& \big(\hatt H^N_{\Om}-zI_\Om\big)^{-1} = \big(H^N_{0,\Om}-zI_\Om\big)^{-1}  \no \\
& \quad - \ol{\big(H^N_{0,\Om}-zI_\Om\big)^{-1} \hatt v}
\Big[I_\Om+\ol{\hatt u\big(H^N_{0,\Om}-zI_\Om\big)^{-1} \hatt v}\,\Big]^{-1}
\hatt u\big(H^N_{0,\Om}-zI_\Om\big)^{-1}. \lb{1.17b}
\end{align}

Then the principal new result proven in \cite{GMZ07} reads as follows:

\begin{theorem} [\cite{GMZ07}] \lb{t1.3}
Assume that $\Om$ satisfies Hypothesis \ref{h1.3a}  and suppose that
$\hatt V\in L^p(\Om;d^nx)$ for some $p$ satisfying $p>4/3$ in the case
$n=2$, and $p>n/2$ in the case $n\geq3$. In addition, let $k\in\bbN$, $k\geq p$ and
$z\in\bbC\big\backslash\big(\si\big(\hatt H_{\Om}^D\big)\cup
\si\big(H_{0,\Om}^D\big) \cup \si\big(H_{0,\Om}^N\big)\big)$. Then,
\begin{align}
&
\frac{\det{}_k\Big(I_{\Om}+\ol{\hatt u\big(H_{0,\Om}^N-zI_{\Om}\big)^{-1} \hatt v}\,\Big)}
{\det{}_k\Big(I_{\Om}+\ol{\hatt u\big(H_{0,\Om}^D-zI_{\Om}\big)^{-1} \hatt v}\,\Big)} \no \\
& \quad = \det{}_k\Big(I_{\dOm} -
\ol{\ga_N\big(\hatt H_{\Om}^D-zI_{\Om}\big)^{-1} \hatt V
\big[\ga_D(H_{0,\Om}^N-\ol{z}I_{\Om})^{-1}\big]^*}\,\Big)
e^{\tr(T_k(z))}   \lb{1.18}  \\
& \quad = \det{}_k\big(M_{\Om}^{D}(z)M_{0,\Om}^{D}(z)^{-1}\big)
e^{\tr(T_k(z))}.   \lb{1.19}
\end{align}
\end{theorem}
Here, ${\det}_k(\cdot)$ denotes the modified Fredholm determinant in
connection with $\cB_k$ perturbations of the identity and $T_k(z)$ is
some trace class operator (cf.\ \cite{GMZ07} for more details). In particular, 
$T_2(z)$ is given by
\begin{equation}
T_2(z)=\ol{\ga_N\big(H_{0,\Om}^D-zI_{\Om}\big)^{-1} V
\big(\hatt H_{\Om}^D-zI_{\Om}\big)^{-1} \hatt V
\big[\ga_D \big(H_{0,\Om}^N-\ol{z}I_{\Om}\big)^{-1}\big]^*},
\end{equation}
where $I_{\Om}$ and $I_{\partial\Om}$ represent the identity operators
in $L^2(\Om; d^n x)$ and  $L^2(\partial\Om; d^{n-1} \sigma)$,
respectively (with $d^{n-1}\sigma$ denoting the surface measure on
$\partial\Om$). For a detailed discussion of the Dirichlet-to-Neumann map  
$M_{\Om}^D(z)$ in $L^2(\dOm; d\sigma^{n-1})$ in connection with local 
interactions $V$ we refer to \cite{GMZ07}. For an extensive list of references 
relevant to the material in \eqref{1.10}--\eqref{1.19} we also refer to \cite{GMZ07}.

Lack of space prevents us from describing a detailed list of papers emphasizing the 
mathematical aspects (and peculiarities) of Schr\"odinger operators with nonlocal interactions (i.e., potentials). Hence, we refer, for instance, to \cite{Al80}, 
\cite{BTV67}, \cite{BTV69}, \cite{BGNS86}, \cite[Ch.\ VIII]{CS89}, \cite{Dr76}, \cite{Dr78}, 
\cite{GR64}, \cite{GR65}, \cite{Ne77}, \cite[Ch.\ 9]{Ne02}, \cite{SWW64}, 
and the list of references cited therein.

Finally, we briefly list most of the notational conventions used
throughout this paper. Let $\cH$ be a separable complex Hilbert
space, $(\cdot,\cdot)_{\cH}$ the scalar product in $\cH$ (linear in
the second factor), and $I_{\cH}$ the identity operator in $\cH$.
Next, let $T$ be a linear operator mapping (a subspace of) a
Banach space into another, with $\dom(T)$ denoting the
domain of $T$. The closure of a closable operator $S$ is
denoted by $\ol S$. The spectrum of a closed linear operator
in $\cH$ will be denoted by $\sigma(\cdot)$. The
Banach spaces of bounded and compact linear operators in $\cH$ are
denoted by $\cB(\cH)$ and $\cB_\infty(\cH)$, respectively. Similarly,
the Schatten--von Neumann (trace) ideals will subsequently be denoted
by $\cB_k(\cH)$, $k\in\bbN$. Analogous notation $\cB(\cH_1,\cH_2)$,
$\cB_\infty (\cH_1,\cH_2)$, etc., will be used for bounded, compact,
etc., operators between two Hilbert spaces $\cH_1$ and $\cH_2$. In
addition, $\tr(T)$ denotes the trace of a trace class operator
$T\in\cB_1(\cH)$ and $\det_{k}(I_{\cH}+S)$ represents the (modified)
Fredholm determinant associated with an operator $S\in\cB_k(\cH)$,
$k\in\bbN$ (for $k=1$ we omit the subscript $1$). Moreover, $\cX_1
\hookrightarrow \cX_2$ denotes the continuous embedding of the Banach
space $\cX_1$ into the Banach space $\cX_2$.

\section{Schr\"odinger Operators with Dirichlet and Neumann boundary conditions and Nonlocal Interactions}
\label{s2}

In this section we recall various properties of Dirichlet,
$H^D_{0,\Om}$, and Neumann, $H^N_{0,\Om}$, Laplacians in
$L^2(\Om;d^n x)$ associated with open sets $\Omega\subset \bbR^n$, $n\in\bbN$, $n\geq 2$, introduced in Hypothesis \ref{h2.1} below. These results have been discussed in detail in \cite{GM08} (see also \cite{GM09}, \cite{GMZ07}). In addition, we introduce the nonlocal Dirichlet and Neumann
Schr\"odinger operators $H^D_{\Om}$ and $H^N_{\Om}$ in
$L^2(\Om;d^n x)$, that is, perturbations of the Dirichlet and Neumann Laplacians 
$H^D_{0,\Om}$ and $H^N_{0,\Om}$ by a (generally, nonlocal) potential $V$ satisfying
Hypothesis \ref{h2.8}.

We start with introducing our assumptions on the set $\Omega$:

\begin{hypothesis}\lb{h2.1}
Let $n\in\bbN$, $n\geq 2$, and assume that $\Om\subset{\bbR}^n$ is
an open, bounded, nonempty Lipschitz domain.
\end{hypothesis}

For more details in the context of the notation used in Hypothesis \ref{h2.1}, and for our notation in connection with Sobolev spaces in the remainder of this paper we refer to \cite[App.\ A]{GM08}.

We introduce the boundary trace operator $\ga_D^0$ (the Dirichlet trace) by
\begin{equation}
\ga_D^0\colon C(\ol{\Om})\to C(\dOm), \quad \ga_D^0 u = u|_\dOm.   \lb{2.5}
\end{equation}
Then there exists a bounded, linear operator $\gamma_D$ (cf., e.g., 
\cite[Theorem 3.38]{Mc00}),
\begin{align}
\begin{split}
& \ga_D\colon H^{s}(\Om)\to H^{s-(1/2)}(\dOm) \hookrightarrow \LdOm,
\quad 1/2<s<3/2, \lb{2.6}  \\
& \ga_D\colon H^{3/2}(\Om)\to H^{1-\varepsilon}(\dOm) \hookrightarrow \LdOm,
\quad \varepsilon \in (0,1), 
\end{split}
\end{align}
whose action is compatible with that of $\ga_D^0$. That is, the two
Dirichlet trace  operators coincide on the intersection of their
domains. We recall that $d^{n-1}\sigma$ denotes the surface measure on
$\dOm$. 

Moreover, we recall that 
\begin{equation}
\ga_D\colon H^{s}(\Om)\to H^{s-(1/2)}(\dOm) \, \text{ is onto for  
$1/2<s<3/2$.} \lb{2.6a}
\end{equation}

While, in the class of bounded Lipschitz subdomains in $\bbR^n$, 
the end-point cases $s=1/2$ and $s=3/2$ of 
$\gamma_D\in\cB\bigl(H^{s}(\Om),H^{s-(1/2)}(\dOm)\bigr)$ fail, we nonetheless
have
\begin{eqnarray}\label{A.62x}
\ga_D\in \cB\big(H^{(3/2)+\eps}(\Om), H^{1}(\dOm)\big), \quad \eps>0.
\end{eqnarray}
See \cite[Lemma\ A.4]{GM08} for a proof. Below we augment this with the following
result: 

\begin{lemma} [\cite{GM08}] \label{Gam-L1}
Assume Hypothesis \ref{h2.1}. Then for each $s>-3/2$, the restriction 
to the boundary operator \eqref{2.5} extends to a linear operator 
\begin{eqnarray}\label{Mam-1}
\gamma_D:\bigl\{u\in H^{1/2}(\Omega)\,\big|\,\Delta u\in H^{s}(\Omega)\bigr\}
\to L^2(\partial\Omega;d^{n-1}\omega),
\end{eqnarray}
is compatible with \eqref{2.6}, and is bounded when 
$\{u\in H^{1/2}(\Omega)\,|\,\Delta u\in H^{s}(\Omega)\bigr\}$ is equipped
with the natural graph norm $u\mapsto \|u\|_{H^{1/2}(\Omega)}
+\|\Delta u\|_{H^{s}(\Omega)}$. In addition, this operator has a
linear, bounded right-inverse $($hence, in particular, it is onto$)$. 

Furthermore, for each $s>-3/2$, the restriction 
to the boundary operator \eqref{2.5} also extends to a linear operator 
\begin{eqnarray}\label{Mam-2}
\gamma_D:\bigl\{u\in H^{3/2}(\Omega)\,\big|\,\Delta u\in H^{1+s}(\Omega)\bigr\}
\to H^1(\partial\Omega), 
\end{eqnarray}
which is compatible with \eqref{2.6}, and is bounded when the set 
$\{u\in H^{3/2}(\Omega)\,|\,\Delta u\in H^{1+s}(\Omega)\bigr\}$ is equipped
with the natural graph norm $u\mapsto \|u\|_{H^{3/2}(\Omega)}
+\|\Delta u\|_{H^{1+s}(\Omega)}$. Once again, this operator has a
linear, bounded right-inverse $($hence, in particular, it is onto$)$. 
\end{lemma}

Next, we introduce the operator $\ga_N$ (the strong Neumann trace) by
\begin{align}\lb{2.7} 
\ga_N = \nu\cdot\ga_D\nabla \colon H^{s+1}(\Om)\to \LdOm, \quad 1/2<s<3/2, 
\end{align}
where $\nu$ denotes the outward pointing normal unit vector to
$\partial\Om$. It follows from \eqref{2.6} that $\ga_N$ is also a
bounded operator. We seek to extend the action of the Neumann trace
operator \eqref{2.7} to other (related) settings. To set the stage, 
assume Hypothesis \ref{h2.1} and recall that the inclusion 
\begin{equation}\lb{inc-1}
\iota:H^s(\Omega)\hookrightarrow \bigl(H^1(\Omega)\bigr)^*,\quad s>-1/2,
\end{equation}
is well-defined and bounded. We then introduce the weak Neumann trace 
operator 
\begin{equation}\lb{2.8}
\wti\ga_N\colon\big\{u\in H^1(\Om)\,\big|\,\Delta u\in H^s(\Om)\big\} 
\to H^{-1/2}(\dOm),\quad s>-1/2,  
\end{equation}
as follows: Given $u\in H^1(\Om)$ with $\Delta u \in H^s(\Om)$ 
for some $s>-1/2$, we set (with $\iota$ as in \eqref{inc-1})
\begin{align} \lb{2.9}
\langle \phi, \wti\ga_N u \rangle_{1/2}
=\int_\Om d^n x\,\ol{\nabla \Phi(x)} \cdot \nabla u(x)  
+ {}_{H^1(\Om)}\langle \Phi, \iota(\Delta u)\rangle_{(H^1(\Om))^*}, 
\end{align}
for all $\phi\in H^{1/2}(\dOm)$ and $\Phi\in H^1(\Om)$ such that
$\ga_D\Phi = \phi$. We note that this definition is
independent of the particular extension $\Phi$ of $\phi$, and that
$\wti\ga_N$ is a bounded extension of the Neumann trace operator
$\ga_N$ defined in \eqref{2.7}.

The end-point case $s=1/2$ of \eqref{2.7} is discussed separately below. 

\begin{lemma} [\cite{GM08}] \label{Neu-tr}
Assume Hypothesis \ref{h2.1}. 
Then the Neumann trace operator \eqref{2.7} also extends to
\begin{eqnarray}\label{MaX-1}
\wti\gamma_N:\bigl\{u\in H^{3/2}(\Omega)\,\big|\,\Delta u\in L^2(\Omega;d^nx)\bigr\}
\to L^2(\partial\Omega;d^{n-1}\omega) 
\end{eqnarray}
in a bounded fashion when the space 
$\{u\in H^{3/2}(\Omega)\,|\,\Delta u\in L^2(\Omega;d^nx)\bigr\}$ is equipped
with the natural graph norm $u\mapsto \|u\|_{H^{3/2}(\Omega)}
+\|\Delta u\|_{L^2(\Omega;d^nx)}$. This extension is compatible 
with \eqref{2.8} and has a linear, bounded, right-inverse $($hence, 
as a consequence, it is onto$)$.  

Moreover, the Neumann trace operator \eqref{2.7} further extends to
\begin{eqnarray}\label{MaX-1U}
\wti\gamma_N:\bigl\{u\in H^{1/2}(\Omega)\,\big|\,\Delta u\in L^2(\Omega;d^nx)\bigr\}
\to H^{-1}(\dOm) 
\end{eqnarray}
in a bounded fashion when the space 
$\{u\in H^{1/2}(\Omega)\,|\,\Delta u\in L^2(\Omega;d^nx)\bigr\}$ is equipped
with the natural graph norm $u\mapsto \|u\|_{H^{1/2}(\Omega)}
+\|\Delta u\|_{L^2(\Omega;d^nx)}$. Once again, this extension is compatible 
with \eqref{2.8} and has a linear, bounded, right-inverse $($thus, in particular,
it is onto$)$.  
\end{lemma}

Next we describe the Dirichlet Laplacian $H_{0,\Om}^D$ in $L^2(\Om;d^nx)$:  

\begin{theorem} [\cite{GM08}] \lb{t2.5}
Assume Hypothesis \ref{h2.1}. Then the Dirichlet Laplacian, 
$H_{0,\Om}^D$, defined by 
\begin{align}
& H_{0,\Om}^D = -\Delta,   \lb{2.39} \\ 
& \dom(H_{0,\Om}^D) = 
\big\{u\in H^1(\Om)\,\big|\, \Delta u \in L^2(\Om;d^n x); \, 
\gamma_D u =0 \text{ in $H^{1/2}(\dOm)$}\big\}   \no \\
& \hspace*{1.67cm}  = \big\{u\in H_0^1(\Om)\,\big|\, \Delta u \in L^2(\Om;d^n x)\big\},  \no 
\end{align}
is self-adjoint and nonnegative $($in fact, strictly positive since $\Om$ is bounded\,$)$ 
in $L^2(\Om;d^nx)$. Moreover,
\begin{equation}
\dom\big((H_{0,\Om}^D)^{1/2}\big) = H^1_0(\Om).   \lb{2.40}
\end{equation}
\end{theorem}

The case of the Neumann Laplacian $H_{0,\Om}^N$ in $L^2(\Om;d^nx)$ is isolated 
next:

\begin{theorem} [\cite{GM08}] \lb{t2.3}
Assume Hypothesis \ref{h2.1}.\ Then the Neumann Laplacian, 
 $H_{0,\Om}^N$, defined by 
\begin{align}
& H_{0,\Om}^N = -\Delta,    \lb{2.20} \\
& \dom\big(H_{0,\Om}^N\big) = 
\big\{u\in H^1(\Om)\,\big|\, \Delta u \in L^2(\Om;d^nx); \, 
\wti\gamma_N u =0 
\text{ in $H^{-1/2}(\dOm)$}\big\},    \no 
\end{align}
is self-adjoint and nonnegative in $L^2(\Om;d^nx)$. Moreover,
\begin{equation}
\dom\big(|H_{0,\Om}^N|^{1/2}\big) = H^1(\Om).   \lb{2.21}
\end{equation}
\end{theorem}

Continuing, we discuss certain regularity results for fractional powers of 
the resolvents of the Dirichlet and Neumann Laplacians in Lipschitz
domains. 

\begin{lemma} [\cite{GM08}] \lb{l2.6}
Assume Hypothesis \ref{h2.1}. In addition, let $q\in [0,1]$ and 
$z\in\bbC\backslash[0,\infty)$. Then,  
\begin{align}
\big(H_{0,\Om}^D-zI_{\Om}\big)^{-q/2},\, \big(H_{0,\Om}^N-zI_{\Om}\big)^{-q/2}
\in\cB\big(\LOm,H^{q}(\Om)\big).   \lb{2.41i}
\end{align}
\end{lemma}

The fractional powers in \eqref{2.41i} (and in subsequent analogous cases) 
are defined via the functional calculus implied by the spectral theorem 
for self-adjoint operators. As discussed in \cite[Lemma A.2]{GLMZ05} in a 
similar context, the key ingredients in proving Lemma \ref{l2.6} are the inclusions
\begin{equation}
\dom\big(H_{0,\Om}^D\big) \subset H^1(\Om), \quad
\dom\big(H_{0,\Om}^N\big) \subset H^1(\Om)
\end{equation}
and real interpolation methods.

Moving on, we now consider mapping properties of powers of the 
resolvents of Neumann Laplacians multiplied (to the left) 
by the Dirichlet boundary trace operator:

\begin{lemma} [\cite{GM08}] \lb{l2.7}
Assume Hypothesis \ref{h2.1}. In addition, let $\eps > 0$ and suppose that 
$z\in\bbC\backslash[0,\infty)$. Then,
\begin{align}
\ga_D \big(H_{0,\Om}^N-zI_{\Om}\big)^{-(1+\eps)/4} \in
\cB\big(\LOm,\LdOm\big).  \lb{2.42}
\end{align}
\end{lemma}

As in \cite[Lemma 6.9]{GLMZ05}, Lemma \ref{l2.7} follows from Lemma \ref{l2.6}
and \eqref{2.6}. We note in passing that \eqref{2.41i} and \eqref{2.42}, extend of course to all $z$ in the resolvent sets of the corresponding operators involved.

Finally, we turn to our assumptions on the (in general, nonlocal) potential $V$ and the 
corresponding definition of Dirichlet and Neumann Schr\"odinger operators 
$H^D_{\Om}$ and $H^N_{\Om}$ in $L^2(\Om; d^n x)$:

\begin{hypothesis} \lb{h2.8}
Suppose that $\Om$ satisfies Hypothesis \ref{h2.1} and assume that
$V \in \cB\big(L^2(\Om;d^nx)\big)$.
\end{hypothesis}

Assuming Hypothesis \ref{h2.8}, we introduce the perturbed operators
$H_{\Om}^D$ and $H_{\Om}^N$ in $\LOm$ by 
\begin{align}
H_{\Om}^D = H_{0,\Om}^D + V, \quad 
\dom\big(H_{\Om}^D\big) = \dom\big(H_{0,\Om}^D\big),   \lb{2.29} \\
H_{\Om}^N = H_{0,\Om}^N + V, \quad 
\dom\big(H_{\Om}^N\big) = \dom\big(H_{0,\Om}^N\big).   \lb{2.30} 
\end{align}
$H_{\Om}^D$ and $H_{\Om}^N$ are self-adjoint in $\LOm$ if and only if 
$V$ is, but self-adjointness will play no role in the remainder of this paper. 

As will be made clear in Remark \ref{r3.2}, it is possible to remove the boundedness assumption on $\Om$ in Hypotheses \ref{h2.1} and \ref{h2.8} and assume that $\dOm$ is compact instead.

\section{Dirichlet and Neumann boundary value problems \\
and Dirichlet-to-Neumann maps} \label{s3}

In this section we review the Dirichlet and Neumann boundary value problems associated with the Helmholtz differential expression $-\Delta - z$ as well as the corresponding differential expression $-\Delta + V - z$ in the presence of a nonlocal potential $V$, both in connection with the open set $\Omega$. In addition, we provide a discussion of Dirichlet-to-Neumann, $M^D_{0,\Om}$, $M^D_{\Om}$, and Neumann-to-Dirichlet maps, $M^N_{0,\Om}$, $M^N_{\Om}$, in $L^2(\partial\Om; d^{n-1}\sigma)$. 

We start with the Helmholtz Dirichlet and Neumann boundary value problems:

\begin{theorem} [\cite{GM08}] \lb{t3.1}
Assume Hypothesis \ref{h2.1}. Then for every 
$f \in H^1(\dOm)$ and $z\in\bbC\big\backslash\si\big(H_{0,\Om}^D\big)$ the
following Dirichlet boundary value problem,
\begin{align} \lb{3.1a}
\begin{cases}
(-\Delta - z)u = 0 \text{ on }\, \Om, \quad u \in H^{3/2}(\Om), \\
\ga_D u = f \text{ on }\, \dOm,
\end{cases}
\end{align}
has a unique solution $u=u_0^D$. This solution satisfies 
$\wti\ga_N u_0^D \in \LdOm$ and there exist constants $C_0^D=C_0^D(\Omega,z)>0$ 
such that
\begin{equation}\lb{MM.xxx}
\|\wti \ga_N u_0^D\|_{\LdOm} \leq C_0^D \|f\|_{H^1(\dOm)}
\end{equation}
as well as 
\begin{equation}
\|u_0^D\|_{H^{3/2}(\Omega)} \leq C_0^D \|f\|_{H^1(\partial\Omega)}.  \lb{3.3aa}
\end{equation}
Similarly, for every $g\in\LdOm$ and
$z\in\bbC\backslash\si\big(H_{0,\Om}^N\big)$ the following Neumann
boundary value problem,
\begin{align} \lb{3.2a}
\begin{cases}
(-\Delta - z)u = 0 \text{ on }\,\Om,\quad u \in H^{3/2}(\Om), \\
\wti\ga_N u = g\text{ on }\,\dOm,
\end{cases}
\end{align}
has a unique solution $u=u_0^N$. This solution satisfies 
$\ga_D u^N_0 \in H^1(\dOm)$ 
and there exist constants $C_0^N=C_0^N(\Omega,z)>0$ such that
\begin{equation}
\|\ga_D u_0^N\|_{H^1(\dOm)} + \|\wti \ga_N u_0^N\|_{\LdOm} \leq 
C_0^N \|g\|_{\LdOm} 
\end{equation}
as well as 
\begin{equation}
\|u_0^N\|_{H^{3/2}(\Omega)} \leq C_0^N \|g\|_{\LdOm}.  \lb{3.4aa}
\end{equation}
In addition,  \eqref{3.1a}--\eqref{3.4aa} imply that the following maps are bounded  
\begin{align}
& \big[\wti \ga_N\big(\big(H^D_{0,\Om}-zI_\Om\big)^{-1}\big)^*\big]^* \in  
\cB\big(H^1(\dOm), H^{3/2}(\Om)\big), \quad 
z\in\bbC\big\backslash\si\big(H^D_{0,\Om}\big),   \lb{3.4ba}
\\
& \big[\ga_D \big(\big(H^N_{0,\Om}-zI_\Om\big)^{-1}\big)^*\big]^* \in 
\cB\big(\LdOm, H^{3/2}(\Om)\big), \quad 
z\in\bbC\big\backslash\si\big(H^N_{0,\Om}\big).   \lb{3.4ca}
\end{align}
Finally, the solutions $u_0^D$ and $u_0^N$ are given by the formulas
\begin{align}
u_0^D (z) &= -\big(\wti \ga_N \big(H_{0,\Om}^D-\ol{z}I_\Om\big)^{-1}\big)^*f,
\lb{3.9a}
\\
u_0^N (z) &= \big(\ga_D \big(H_{0,\Om}^N-\ol{z}I_\Om\big)^{-1}\big)^*g. 
\lb{3.10a}
\end{align}
\end{theorem}

\begin{remark} \lb{r3.2}
It is possible to remove the boundedness assumption on $\Om$ in 
Hypotheses \ref{h2.1} and \ref{h2.8} and assume that $\dOm$ is compact instead. 

Consider, for example, the case of the Dirichlet boundary value problem \eqref{3.1a}, 
this time formulated for an unbounded Lipschitz domain 
$\Omega\subset{\mathbb{R}}^n$ 
with a compact boundary. We claim that the same type of well-posedness statement 
as in Theorem~\ref{t3.1} holds in this setting as well. To see this, consider first
the auxiliary problem 
\begin{align} \lb{3.1aX}
\begin{cases}
(-\Delta - z)u = 0 \text{ on }\, \Om, \quad u \in H^{1}(\Om), \\
\ga_D u = f \, \text{ on } \, \dOm,
\end{cases}
\end{align}
which we claim has a unique solution whenever 
$f \in H^{1/2}(\dOm)$ and $z\in\bbC\big\backslash\si\big(H_{0,\Om}^D\big)$. 
In addition, there exists a constant $C_0^D=C_0^D(\Omega,z)>0$  such that the solution 
$u=u_0^D$ of \eqref{3.1aX} satisfies 
\begin{equation}
\|u_0^D\|_{H^{1}(\Omega)} \leq C_0^D \|f\|_{H^{1/2}(\partial\Omega)}.  \lb{3.3aaX}
\end{equation}
To justify this claim, one first observes that there exists a constant $C=C(\Omega)>0$ 
with the property that, given any $f \in H^{1/2}(\dOm)$, it is possible to select a 
function $w\in H^1(\Om)$ such that  
\begin{equation}\lb{MM.1}
\gamma_Dw=f \, \text{ and } \, \|w\|_{H^1(\Om)}\leq C\|f\|_{H^{1/2}(\dOm)}.
\end{equation}
Granted this and having fixed $z\in\bbC\big\backslash\si\big(H_{0,\Om}^D\big)$, 
a solution $u$ for \eqref{3.1aX} can be found in the form 
\begin{equation}\lb{MM.2}
u=u^D_0:=w+ \Big(\widetilde{H_{0,\Om}^D} -zI_{\Om}\Big)^{-1}[(\Delta+z)w], 
\end{equation}
where $\Big(\widetilde{H_{0,\Om}^D} - zI_{\Om}\Big)^{-1}
\in\cB\bigl(H^{-1}(\Om),H^1(\Om)\bigr)$. 
It is then clear that the function $u^D_0$ constructed in \eqref{MM.2} solves 
\eqref{3.1aX}
and satisfies \eqref{3.3aaX}. The uniqueness of such a solution is then a consequence 
of the fact that $z\in\bbC\big\backslash\si\big(H_{0,\Om}^D\big)$. Here 
$\widetilde{H_{0,\Om}^D}$ denotes an extension of the self-adjoint operator 
$H_{0,\Om}^D$  in $L^2(\Om; d^n x)$ familiar from the theory of densely defined, closed  sesquilinear forms bounded from below and their associated self-adjoint operators as discussed in detail in \cite{GM08} (cf.\ App.\ B, in particular, (B.11)--(B.19)). 

Having settled the issue of the well-posedness of \eqref{3.1aX}, we now proceed to 
show that, in the case of an unbounded Lipschitz domain with a compact boundary, 
one has the regularity statement 
\begin{equation}\lb{MM.3}
f\in H^1(\dOm)\hookrightarrow H^{1/2}(\dOm) \, \text{ implies } \, u^D_0\in H^{3/2}(\Om).
\end{equation}
To see this, in addition to $z\in\bbC\big\backslash\si\big(H_{0,\Om}^D\big)$, 
pick a complex number $z_0\in\bbC\backslash\bbR$. Then for every $f\in H^1(\dOm)$ 
we know that $u=u^D_0$ belongs to $H^1(\Om)$ and our goal is to show that, in fact, 
$u=u^D_0\in H^{3/2}(\Om)$. This is done using a suitable representation for $u$, 
namely
\begin{equation}\lb{MM.4}
u=v+w
\end{equation}
where we have set
\begin{equation}\lb{MM.5}
v:=\cS_{z_0}\big[S_{z_0}^{-1}(f-\gamma_Dw)\big],\quad 
w:=(z-z_0)\bigl(E_n (z_0;\cdot)\ast u). 
\end{equation}
Here $E_n(z_0;x)$ is the fundamental solution of the Helmholtz differential expression 
$(-\Delta -z_0)$ in $\bbR^n$, $n\in\bbN$, $n\geq 2$, that is,
\begin{align}
& E_n(z_0;x) = \begin{cases}
(i/4) \big(2\pi |x|/z_0^{1/2}\big)^{(2-n)/2} H^{(1)}_{(n-2)/2} 
\big(z_0^{1/2}|x|\big), & n\geq 2, \; z_0\in\bbC\backslash \{0\}, \\
\f{-1}{2\pi} \ln(|x|), & n=2, \; z_0=0, \\ 
\f{1}{(n-2)\omega_{n-1}}|x|^{2-n}, & n\geq 3, \; z_0=0, 
\end{cases}   \no \\
& \hspace*{6.8cm} \Im\big(z_0^{1/2}\big)\geq 0,\; x\in\bbR^n\backslash\{0\}, 
 \lb{3.18} 
\end{align}
with $H^{(1)}_{\nu}(\dott)$ denoting the Hankel function of the first kind 
with index $\nu\geq 0$ (cf.\ \cite[Sect.\ 9.1]{AS72}), and 
\begin{equation}\lb{MM.6}
(\cS_{z_0})h(x):=\int_{\dOm}d^{n-1}\sigma(y)\,E_n (z_0;(x-y)h(y),  
\quad x\in\Om,
\end{equation}
is the so-called single layer potential operator for 
$(-\Delta-z_0)$ in ${\mathbb{R}}^n$, and finally, 
\begin{equation}
S_{z_0}:=\gamma_D \, \cS_{z_0}.
\end{equation} 
From \cite{GM08} we know that 
$S_{z_0}\in\cB\bigl(L^2(\dOm;d^{n-1}\sigma),H^1(\dOm)\bigr)$ is an isomorphism 
with $S_{z_0}^{-1}\in\cB\bigl(H^1(\dOm),L^2(\dOm;d^{n-1}\sigma)\bigr)$ and, if 
$\psi\in C^\infty_0({\mathbb{R}}^n)$ is identically one in an open neighborhood of 
$\overline{\Omega}$, and $M_\psi$ denotes the operator of multiplication by $\psi$, then 
\begin{equation}
M_\psi \, \cS_{z_0}\in\cB\bigl(L^2(\dOm;d^{n-1}\sigma),H^{3/2}(\Om)\bigr).
\end{equation} 
Furthermore, it has been observed in \cite{GM08} that for any multi-index $\alpha$, the 
function $\partial^\alpha E_{n,z_0}(x)$ decays exponentially at infinity (here, the fact that 
$\Im(z_0)\neq 0$ is used). In turn, this readily yields that $w\in H^2(\Om)$ (hence, in 
particular, $\gamma_Dw\in H^1(\dOm)$), and $v\in H^{3/2}(\Om)$. Consequently, one concludes that $u$ belongs to $H^{3/2}(\Om)$ and its norm in this space is majorized by a fixed multiple of $\|f\|_{H^1(\dOm)}$. 

Having establish the existence of a unique solution $u$ for \eqref{3.1a} in the case when 
$\Omega$ is an unbounded Lipschitz domain with compact boundary, then 
\eqref{MM.xxx} follows from this as in the case of bounded domains. 

The reasoning for the Neumann problem \eqref{3.2a} is very similar, and we omit it. 
\end{remark}

By employing a perturbative approach, one extends Theorem \ref{t3.1} in connection with the Helmholtz differential expression $-\Delta - z$ on $\Om$ to the case of a 
Schr\"odinger operator corresponding to $-\Delta + V - z$ on $\Om$, with $V$ a 
generally nonlocal interaction.

\begin{theorem} \lb{t3.3}
Assume Hypothesis \ref{h2.8}. Then for every 
$f \in H^1(\dOm)$ and $z\in\bbC\big\backslash\si\big(H_{\Om}^D\big)$ the
following Dirichlet boundary value problem,
\begin{align} \lb{3.1}
\begin{cases}
(-\Delta + V - z)u = 0 \text{ on }\, \Om, \quad u \in H^{3/2}(\Om), \\
\ga_D u = f \text{ on }\, \dOm,
\end{cases}
\end{align}
has a unique solution $u=u^D$. This solution satisfies 
$\wti\ga_N u^D \in \LdOm$ and there exist constants $C^D=C^D(\Omega,z)>0$ 
such that
\begin{equation}
\|\wti \ga_N u^D\|_{\LdOm} \leq C^D \|f\|_{H^1(\dOm)}
\end{equation}
as well as 
\begin{equation}
\|u^D\|_{H^{3/2}(\Omega)} \leq C^D \|f\|_{H^1(\partial\Omega)}.  \lb{3.3a}
\end{equation}
Similarly, for every $g\in\LdOm$ and
$z\in\bbC\backslash\si\big(H_{\Om}^N\big)$ the following Neumann
boundary value problem,
\begin{align} \lb{3.2}
\begin{cases}
(-\Delta + V - z)u = 0 \text{ on }\,\Om,\quad u \in H^{3/2}(\Om), \\
\wti\ga_N u = g\text{ on }\,\dOm,
\end{cases}
\end{align}
has a unique solution $u=u^N$. This solution satisfies 
$\ga_D u^N \in H^1(\dOm)$ 
and there exist constants $C^N=C^N(\Omega,z)>0$ such that
\begin{equation}
\|\ga_D u^N\|_{H^1(\dOm)} + \|\wti \ga_N u^N\|_{\LdOm} \leq 
C^N \|g\|_{\LdOm} 
\end{equation}
as well as 
\begin{equation}
\|u^N\|_{H^{3/2}(\Omega)} \leq C^N \|g\|_{\LdOm}.  \lb{3.4a}
\end{equation}
In addition,  \eqref{3.1}--\eqref{3.4a} imply that the following maps are bounded  
\begin{align}
& \big[\wti \ga_N\big(\big(H^D_{\Om}-zI_\Om\big)^{-1}\big)^*\big]^* \in  
\cB\big(H^1(\dOm), H^{3/2}(\Om)\big), \quad 
z\in\bbC\big\backslash\si\big(H^D_{\Om}\big),   \lb{3.4b}
\\
& \big[\ga_D \big(\big(H^N_{\Om}-zI_\Om\big)^{-1}\big)^*\big]^* \in 
\cB\big(\LdOm, H^{3/2}(\Om)\big), \quad 
z\in\bbC\big\backslash\si\big(H^N_{\Om}\big).   \lb{3.4c}
\end{align}
Finally, the solutions $u^D$ and $u^N$ are given by the formulas
\begin{align}
u^D (z) &= -\big(\wti \ga_N \big(H_{\Om}^D-\ol{z}I_\Om\big)^{-1}\big)^*f,
\lb{3.3}
\\
u^N (z) &= \big(\ga_D \big(H_{\Om}^N-\ol{z}I_\Om\big)^{-1}\big)^*g. 
\lb{3.4}
\end{align}
\end{theorem}
\begin{proof}
One can follow the proof of \cite[Theorems\ 3.2 and 3.6]{GM08}, using the fact that 
the functions
\begin{align}
u^D (z) &= u_0^D (z) - \big(H_\Om^D-zI_\Om\big)^{-1} V u_0^D (z), \lb{3.13}
\\
u^N (z)&= u_0^N (z) - \big(H_\Om^N-zI_\Om\big)^{-1} V u_0^N (z),  \lb{3.14}
\end{align}
with $u_0^D, u_0^N$ given by Theorem \ref{t3.1}, satisfy \eqref{3.9a}
and \eqref{3.10a}, respectively. 
\end{proof}

Assuming Hypothesis \ref{h2.1}, we now introduce the
Dirichlet-to-Neumann map $M_{0,\Om}^{D}(z)$ 
associated with $(-\Delta-z)$ on $\Om$, following \cite{GMZ07}, 
\begin{align}
M_{0,\Om}^{D}(z) \colon
\begin{cases}
H^1(\dOm) \to \LdOm,
\\
\hspace*{10mm} f \mapsto -\wti\ga_N u_0^D,
\end{cases}  \quad z\in\bbC\big\backslash\si\big(H_{0,\Om}^D\big), \lb{3.20}
\end{align}
where $u_0^D$ is the unique solution of
\begin{align}
(-\Delta-z)u_0^D = 0 \,\text{ on }\Om, \quad u_0^D\in
H^{3/2}(\Om), \quad \ga_D u_0^D = f \,\text{ on }\dOm. 
\end{align}
Similarly, assuming Hypothesis \ref{h2.8}, we introduce the
Dirichlet-to-Neumann map $M_\Om^{D}(z)$, 
associated with $(-\Delta+V-z)$ on $\Om$, by 
\begin{align}
M_\Om^{D}(z) \colon
\begin{cases}
H^1(\dOm) \to \LdOm,
\\
\hspace*{10mm} f \mapsto -\wti\ga_N u^D,
\end{cases}  \quad z\in\bbC\big\backslash\si\big(H_{\Om}^D\big), \lb{3.22}
\end{align}
where $u^D$ is the unique solution of
\begin{align}
(-\Delta+V-z)u^D = 0 \,\text{ on }\Om,  \quad u^D \in
H^{3/2}(\Om), \quad \ga_D u^D= f \,\text{ on }\dOm.
\end{align}
By Theorems \ref{t3.1} and \ref{t3.3} one obtains
\begin{equation}
M_{0,\Om}^{D}(z), M_\Om^{D}(z) \in \cB\big(H^1(\partial\Om), \LdOm \big).
\end{equation}

In addition, assuming Hypothesis \ref{h2.1}, we introduce the Neumann-to-Dirichlet map  $M_{0,\Om}^{N}(z)$ associated with $(-\Delta-z)$ on $\Om$, as follows,
\begin{align}
M_{0,\Om}^{N}(z) \colon \begin{cases} \LdOm \to H^1(\dOm),
\\
\hspace*{20.5mm} g \mapsto \ga_D u_0^N, \end{cases}  \quad
z\in\bbC\big\backslash\si\big(H_{0,\Om}^N\big), \lb{3.24}
\end{align}
where $u_0^N$ is the unique solution of
\begin{align}
(-\Delta-z)u_0^N = 0 \,\text{ on }\Om, \quad u_0^N\in
H^{3/2}(\Om), \quad \wti\ga_N u_0^N = g \,\text{ on }\dOm. 
\end{align}
Similarly, assuming Hypothesis \ref{h2.8}, we introduce the Neumann-to-Dirichlet map 
$M_\Om^{N}(z)$ associated with $(-\Delta+V-z)$ on $\Om$ by 
\begin{align}
M_\Om^{N}(z) \colon \begin{cases}
\LdOm \to H^1(\dOm),
\\
\hspace*{20.5mm} g \mapsto \ga_D u^N,
\end{cases}  \quad
z\in\bbC\big\backslash\si\big(H_{\Om}^N\big), \lb{3.26}
\end{align}
where $u^N$ is the unique solution of
\begin{align}
(-\Delta+V-z)u^N = 0 \,\text{ on }\Om,  \quad u^N \in
H^{3/2}(\Om), \quad \wti\ga_N u^N= g \,\text{ on }\dOm.
\end{align}
Again, by Theorems \ref{t3.1} and \ref{t3.3} one obtains
\begin{equation}
M_{0,\Om}^{N}(z), M_\Om^{N}(z) \in \cB\big(\LdOm, H^1(\partial\Om) \big).
\end{equation}
In particular, $M^{N}_{\Om}(z)$, $z\in\bbC\big\backslash\si\big(H_{\Om}^N\big)$, are compact operators in $L^2(\partial\Om; d^{n-1}\sigma)$ since 
$H^1(\partial\Om)$ embeds compactly into $L^2(\partial\Om; d^{n-1}\sigma)$ 
(cf.\ \cite[Proposition\ 2.4]{MM07}).  

Moreover, under the assumption of Hypothesis \ref{h2.1} for $M_{0,\Om}^D(z)$ and 
$M_{0,\Om}^N(z)$, and under the assumption of Hypothesis \ref{h2.8} for 
$M_{\Om}^D(z)$ and $M_{\Om}^N(z)$, one infers the following equalities: 
\begin{align}
M_{0,\Om}^{N}(z) &= - M_{0,\Om}^{D}(z)^{-1}, \quad
z\in\bbC\big\backslash\big(\si\big(H_{0,\Om}^D\big)\cup\si\big(H_{0,\Om}^N\big)\big),
\lb{3.28}
\\
M_{\Om}^{N}(z) &= - M_{\Om}^{D}(z)^{-1}, \quad
z\in\bbC\big\backslash\big(\si\big(H_{\Om}^D\big)\cup\si\big(H_{\Om}^N\big)\big),
\lb{3.29}
\intertext{and} 
M^{D}_{0,\Om}(z) &= \wti\gamma_N\big[\wti \gamma_N
\big(\big(H^D_{0,\Om} - zI_\Om\big)^{-1}\big)^*\big]^*, \quad
z\in\bbC\big\backslash\si\big(H_{0,\Om}^D\big), \lb{3.30}
\\
M^{D}_{\Om}(z) &= \wti\gamma_N\big[\wti \gamma_N \big(\big(H^D_{\Om} -
zI_\Om\big)^{-1}\big)^*\big]^*, \quad
z\in\bbC\big\backslash\si\big(H_{\Om}^D\big),
\lb{3.31}
\\
M^{N}_{0,\Om}(z) &= \gamma_D\big[\gamma_D
\big(\big(H^N_{0,\Om} - zI_\Om\big)^{-1}
\big)^*\big]^*, \quad
z\in\bbC\big\backslash\si\big(H_{0,\Om}^N\big), \lb{3.32}
\\
M^{N}_{\Om}(z) &= \gamma_D\big[\gamma_D
\big(\big(H^N_{\Om} - zI_\Om\big)^{-1}\big)^*\big]^*, \quad
z\in\bbC\big\backslash\si\big(H_{\Om}^N\big).
\lb{3.33}
\end{align}

Next, we note the following auxiliary result, which will play a
crucial role in Theorem \ref{t4.2}, the principal result of this
paper:

\begin{lemma} \lb{l3.5}
Assume Hypothesis \ref{h2.8}. Then the following
identities hold,
\begin{align}
M_{0,\Om}^D(z) - M_\Om^D(z) &= \wti\gamma_N
\big(H^D_{\Om}-zI_\Om\big)^{-1} V \big[\wti \gamma_N
\big(\big(H^D_{0,\Om}-zI_\Om\big)^{-1}\big)^*\big]^*, \no
\\
&\hspace*{3.1cm}
z\in\bbC\big\backslash\big(\si\big(H_{0,\Om}^D\big)\cup\si\big(H_{\Om}^D\big)\big),
\lb{3.35}
\\
M_\Om^D(z) M_{0,\Om}^D(z)^{-1} &= I_\dOm - \wti\gamma_N
\big(H^D_{\Om}-zI_\Om\big)^{-1} V \big[\gamma_D
\big(\big(H^N_{0,\Om}-zI_\Om\big)^{-1}\big)^*\big]^*, \no
\\
&\hspace*{2.45cm}
z\in\bbC\big\backslash\big(\si\big(H_{0,\Om}^D\big)\cup\si\big(H_{\Om}^D\big)
\cup\si\big(H_{0,\Om}^N\big)\big). \lb{3.36}
\end{align}
\end{lemma}

For the proof of Lemma \ref{l3.5} one can follow the corresponding proof of 
Lemma\ 3.6 in \cite{GMZ07} step by step. 

We note that the right-hand sides (and hence the left-hand sides) of \eqref{3.35}
and \eqref{3.36} permit of course an analytic continuation with respect to $z$ 
as long as $z$ varies in the resolvent sets of the corresponding operators involved.

Again we note that due to the reasoning in Remark \ref{r3.2} it is possible to remove the boundedness assumption on $\Om$ in Hypotheses \ref{h2.1} and \ref{h2.8} and assume that $\dOm$ is compact throughout this section.

\section{A Multi-Dimensional Variant of a Formula due to Jost and Pais in the 
Presence of Nonlocal Interactions}
\label{s4}

In this final section we prove our principal new result, a variant of a multi-dimensional Jost--Pais formula in the presence of nonlocal interactions as discussed in the introduction.

We start by providing an elementary comment on determinants which, however, lies at the heart of the matter of our principal result, Theorem \ref{t4.2}: Suppose $A \in \cB(\cH_1, \cH_2)$, $B \in \cB(\cH_2, \cH_1)$ with $A B \in \cB_1(\cH_2)$ and $B A \in \cB_1(\cH_1)$. Then,
\begin{equation}
\det (I_{\cH_2}-AB) = \det (I_{\cH_1}-BA).   \lb{4.0}
\end{equation}
Equation \eqref{4.0} follows from the fact that all nonzero eigenvalues of $AB$ and $BA$ coincide including their algebraic multiplicities. 

In particular, $\cH_1$ and $\cH_2$ may have different dimensions.
Especially, one of them may be infinite and the other finite, in
which case one of the two determinants in \eqref{4.0} reduces to a
finite determinant. This case indeed occurs in the original
one-dimensional case studied by Jost and Pais \cite{JP51} as
described in detail in \cite{GM03} and the references therein.
In the proof of Theorem \ref{t4.1} below, the role of $\cH_1$ and
$\cH_2$ will be played by $L^2(\Om; d^n x)$ and
$L^2(\partial\Om;d^{n-1} \sigma)$, respectively.

Next, we introduce the appropriate additional trace class assumption on the nonlocal potential $V$:

\begin{hypothesis} \lb{h4.0}
Suppose that $\Om$ satisfies Hypothesis \ref{h2.1}  and assume that
$V \in \cB_1\big(L^2(\Om;d^nx)\big)$.
\end{hypothesis}

Since $V\in \cB_1\big(L^2(\Om;d^nx)\big)$ we may assume (without loss of generality) that 
\begin{equation}
V = v u, \, \text{ where } \, u, v \in \cB_2\big(L^2(\Om;d^nx)\big)
\end{equation}
and we fix the pair $(u,v)$ associated with $V$ in the following. Thus, one infers 
(for $z\in\bbC\backslash [0,\infty)$)
\begin{align}
& u\big(H_{0,\Om}^D-zI_{\Om}\big)^{-1/2}, \,
\big(H_{0,\Om}^D-zI_{\Om}\big)^{-1/2}v \in\cB_{2}\big(\LOm\big), \lb{2.31} \\
& u\big(H_{0,\Om}^N-zI_{\Om}\big)^{-1/2}, \,
 \big(H_{0,\Om}^N-zI_{\Om}\big)^{-1/2}v \in\cB_{2}\big(\LOm\big), \lb{2.32} \\
& u\big(H_{0,\Om}^D-zI_{\Om}\big)^{-1}v, \,
u\big(H_{0,\Om}^N-zI_{\Om}\big)^{-1}v \in\cB_1\big(\LOm\big),   \lb{2.35}
\end{align}
and hence obtains the resolvent identities (still for $z\in\bbC\backslash [0,\infty)$) 
\begin{align}
& \big(H^D_{\Om}-zI_\Om\big)^{-1} = \big(H^D_{0,\Om}-zI_\Om\big)^{-1}  \no \\
& \quad - \big(H^D_{0,\Om}-zI_\Om\big)^{-1}v
\Big[I_\Om+ u\big(H^D_{0,\Om}-zI_\Om\big)^{-1}v\,\Big]^{-1}
u\big(H^D_{0,\Om}-zI_\Om\big)^{-1},     \lb{3.47a} \\ 
& \big(H^N_{\Om}-zI_\Om\big)^{-1} = \big(H^N_{0,\Om}-zI_\Om\big)^{-1}  \no \\
& \quad - \big(H^N_{0,\Om}-zI_\Om\big)^{-1}v 
\Big[I_\Om+ u\big(H^N_{0,\Om}-zI_\Om\big)^{-1}v \,\Big]^{-1}
u\big(H^N_{0,\Om}-zI_\Om\big)^{-1}.    \lb{3.48a}
\end{align}
We note in passing that \eqref{2.31}--\eqref{3.47a}, \eqref{3.48a} extend 
of course to all $z$ in the resolvent sets of the corresponding operators involved.

We continue by proving an extension of a result in \cite{GLMZ05} to arbitrary space dimensions:

\begin{theorem} \lb{t4.1}
Assume Hypothesis \ref{h4.0} and
$z\in\bbC\big\backslash\big(\si\big(H_{\Om}^D\big)\cup
\si\big(H_{0,\Om}^D\big) \cup \si\big(H_{0,\Om}^N\big)\big)$. Then,
\begin{align}
\wti \ga_N\big(H_{\Om}^D-zI_{\Om}\big)^{-1}V \Big[\ga_D
\big(H_{0,\Om}^N-\ol{z}I_{\Om}\big)^{-1}\Big]^*
\in\cB_1\big(\LdOm\big)   \lb{4.2}
\end{align}
and
\begin{align}
\begin{split} 
&
\frac{\det{}\Big(I_{\Om}+u\big(H_{0,\Om}^N-zI_{\Om}\big)^{-1}v\,\Big)}
{\det{}\Big(I_{\Om}+u\big(H_{0,\Om}^D-zI_{\Om}\big)^{-1}v\,\Big)}     \\
&\quad = \det{}\Big(I_{\dOm} -
\wti \ga_N\big(H_{\Om}^D-zI_{\Om}\big)^{-1}V \Big[\ga_D
\big(H_{0,\Om}^N-\ol{z}I_{\Om}\big)^{-1}\Big]^* \, \Big).   \lb{4.3}
\end{split} 
\end{align}
\end{theorem}
\begin{proof}
From the outset we note that the left-hand side of \eqref{4.3} is
well-defined by \eqref{2.35}. Let
$z\in\bbC\big\backslash\big(\si\big(H_{\Om}^D\big) \cup
\si\big(H_{0,\Om}^D\big) \cup \si\big(H_{0,\Om}^N\big)\big)$.

Next, we introduce
\begin{equation} \lb{4.7}
K_D(z)=- u\big(H_{0,\Om}^D-zI_{\Om}\big)^{-1}v, \quad
K_N(z)=- u\big(H_{0,\Om}^N-zI_{\Om}\big)^{-1}v
\end{equation}
and note that  
\begin{align}
[I_{\Om}-K_D(z)]^{-1} \in\cB\big(\LOm\big), \quad
z\in\bbC\big\backslash\big(\si\big(H_{\Om}^D\big)\cup\si\big(H_{0,\Om}^D\big)\big). 
\lb{4.8}
\end{align}
Hence one concludes that 
\begin{align}
\begin{split}
&\frac{\det\Big(I_{\Om}+ u\big(H_{0,\Om}^N-zI_{\Om}\big)^{-1}v\,\Big)}
{\det\Big(I_{\Om}+u\big(H_{0,\Om}^D-zI_{\Om}\big)^{-1}v\,\Big)}
=
\frac{\det\big(I_{\Om}-K_N(z)\big)}{\det\big(I_{\Om}-K_D(z)\big)}   \\
&\quad =
\det\big(I_{\Om}-(K_N(z)-K_D(z))[I_{\Om}-K_D(z)]^{-1}\big).   \lb{4.12}
\end{split} 
\end{align}

Using \cite[Lemma\ A.3]{GMZ07} (an extension of a
result of Nakamura \cite[Lemma\ 6]{Na01}) and \cite[Remark\ A.5]{GMZ07}, 
one finds
\begin{align}
K_N(z)-K_D(z) &= u\big[\big(H_{0,\Om}^D-zI_{\Om}\big)^{-1}-
\big(H_{0,\Om}^N -zI_{\Om}\big)^{-1}\big]v  \no
\\ &=
u\Big[\ga_D \big(H_{0,\Om}^N-\ol{z}I_{\Om}\big)^{-1}\Big]^* \,
\wti \ga_N \big(H_{0,\Om}^D -zI_{\Om}\big)^{-1}v  \no
\\ &=
\Big[\,\ga_D
\big(H_{0,\Om}^N-\ol{z}I_{\Om}\big)^{-1} u^* \,\Big]^* \,
\wti \ga_N \big(H_{0,\Om}^D -zI_{\Om}\big)^{-1}v.     \lb{4.13}
\end{align}
Insertion of \eqref{4.13} into \eqref{4.12} then yields
\begin{align} \lb{4.14}
&\frac{\det\Big(I_{\Om}+ u\big(H_{0,\Om}^N-zI_{\Om}\big)^{-1}v\,\Big)}
{\det\Big(I_{\Om}+ u\big(H_{0,\Om}^D-zI_{\Om}\big)^{-1}v\,\Big)}
\no
\\
&\quad = \det\Big(I_{\Om} - \Big[\, \ga_D
\big(H_{0,\Om}^N-\ol{z}I_{\Om}\big)^{-1} u^*\,\Big]^* \wti \ga_N
\big(H_{0,\Om}^D-zI_{\Om}\big)^{-1}v   \no \\
& \hspace*{1.6cm} \times  
\Big[I_{\Om}+ u\big(H_{0,\Om}^D-zI_{\Om}\big)^{-1}v \,\Big]^{-1}\Big). 
\end{align}
Since 
\begin{align}
\ga_D \big(H_{0,\Om}^N-\ol{z}I_{\Om}\big)^{-1} u^*
&\in\cB_{2}\big(\LOm,\LdOm\big),
\\
\wti \ga_N\big(H_{0,\Om}^D-zI_{\Om}\big)^{-1}v 
&\in\cB_{2}\big(\LOm,\LdOm\big),
\end{align}
one concludes that 
\begin{align}
\Big[\, \ga_D
\big(H_{0,\Om}^N-\ol{z}I_{\Om}\big)^{-1} u^* \,\Big]^* \wti \ga_N
\big(H_{0,\Om}^D-zI_{\Om}\big)^{-1}v &\in \cB_1\big(\LOm\big),
\\
\wti \ga_N \big(H_{0,\Om}^D-zI_{\Om}\big)^{-1}v \Big[\, \ga_D
\big(H_{0,\Om}^N-\ol{z}I_{\Om}\big)^{-1} u^* \,\Big]^* &\in
\cB_1\big(\LdOm\big).
\end{align}
Then, using \eqref{4.8}, 
one applies the idea expressed in formula \eqref{4.0} and rearranges
the terms in \eqref{4.14} as follows:
\begin{align} \lb{4.20}
&\frac{\det\Big(I_{\Om}+ u\big(H_{0,\Om}^N-zI_{\Om}\big)^{-1}v \,\Big)}
{\det\Big(I_{\Om}+ u\big(H_{0,\Om}^D-zI_{\Om}\big)^{-1}v \,\Big)}   \no 
\\
&\quad = \det\Big(I_{\dOm} - \wti \ga_N
\big(H_{0,\Om}^D-zI_{\Om}\big)^{-1}v 
\Big[I_{\Om}+ u\big(H_{0,\Om}^D-zI_{\Om}\big)^{-1}v \,\Big]^{-1}  \no \\
& \hspace*{1.5cm} \times \Big[\, \ga_D
\big(H_{0,\Om}^N-\ol{z}I_{\Om}\big)^{-1} u^* \,\Big]^*\Big)   \no \\
&\quad = \det\Big(I_{\dOm} - \wti \ga_N
\big(H_{0,\Om}^D-zI_{\Om}\big)^{-1}v 
\Big[I_{\Om}+ u\big(H_{0,\Om}^D-zI_{\Om}\big)^{-1}v \,\Big]^{-1}  \no \\
& \hspace*{1.5cm} \times u \Big[\, \ga_D
\big(H_{0,\Om}^N-\ol{z}I_{\Om}\big)^{-1}\Big]^*\Big). 
\end{align}

Finally, using
\begin{align}
\big(H_{\Om}^D-zI_{\Om}\big)^{-1} v =
\big(H_{0,\Om}^D-zI_{\Om}\big)^{-1} v \Big[I_{\Om}+ 
u\big(H_{0,\Om}^D-zI_{\Om}\big)^{-1} v \,\Big]^{-1},   \lb{4.29a}
\end{align}
proves \eqref{4.3}.
\end{proof}

Given these preparations, we are ready for the principal result of this paper, the multi-dimensional analog of Theorem \ref{t1.2} in the context of nonlocal interactions:

\begin{theorem} \lb{t4.2}
Assume Hypothesis \ref{h4.0} and
$z\in\bbC\big\backslash\big(\si\big(H_{\Om}^D\big)\cup
\si\big(H_{0,\Om}^D\big) \cup \si\big(H_{0,\Om}^N\big)\big)$. Then,
\begin{align}
\begin{split}
& M_{\Om}^{D}(z)M_{0,\Om}^{D}(z)^{-1} - I_{\partial\Om}  \\
& \quad = - \wti \ga_N\big(H_{\Om}^D-zI_{\Om}\big)^{-1} V
\Big[\ga_D \big(H_{0,\Om}^N-\ol{z}I_{\Om}\big)^{-1}\Big]^* \in
\cB_1\big(L^2(\partial\Om; d^{n-1}\sigma)\big)
\end{split}
\end{align}
and
\begin{align}
&
\frac{\det\Big(I_{\Om}+ u\big(H_{0,\Om}^N-zI_{\Om}\big)^{-1}v\,\Big)}
{\det\Big(I_{\Om}+ u\big(H_{0,\Om}^D-zI_{\Om}\big)^{-1}v\,\Big)} \no \\
& \quad = \det\Big(I_{\dOm} -
\wti \ga_N\big(H_{\Om}^D-zI_{\Om}\big)^{-1} V
\Big[\ga_D\big(H_{0,\Om}^N-\ol{z}I_{\Om}\big)^{-1}\Big]^* \,\Big)  \lb{4.30}  \\
& \quad = \det\big(M_{\Om}^{D}(z)M_{0,\Om}^{D}(z)^{-1}\big).   \lb{4.31}
\end{align}
\end{theorem}
\begin{proof}
The result follows from combining Lemma \ref{l3.5} and Theorem \ref{t4.1}.
\end{proof}

\begin{remark}  \lb{r4.3}
Assume Hypothesis \ref{h4.0} and
$z\in\bbC\big\backslash\big(\si\big(H_{\Om}^N\big)\cup
\si\big(H_{0,\Om}^D\big) \cup \si\big(H_{0,\Om}^N\big)\big)$. Then,
\begin{align}
\begin{split}
& M_{0,\Om}^{N}(z)^{-1}M_{\Om}^{N}(z) - I_{\partial\Om}     \\ 
& \quad =
\wti \ga_N \big(H_{0,\Om}^D-zI_{\Om}\big)^{-1}V \Big[\ga_D \big(\big(H_{\Om}^N-z
I_{\Om}\big)^{-1}\big)^*\Big]^* \in
\cB_1\big(L^2(\partial\Om; d^{n-1}\sigma)\big)   \lb{4.32}
\end{split} 
\end{align}
and one can also prove the following analog of \eqref{4.30} and \eqref{4.31}:
\begin{align}
&\frac{\det\Big(I_{\Om}+ u\big(H_{0,\Om}^D-zI_{\Om}\big)^{-1}v \,\Big)}
{\det\Big(I_{\Om}+ u\big(H_{0,\Om}^N-zI_{\Om}\big)^{-1}v \,\Big)} \no \\
&\quad = \det\Big(I_{\dOm} +
\wti \ga_N \big(H_{0,\Om}^D-zI_{\Om}\big)^{-1}V \Big[\ga_D \big(\big(H_{\Om}^N-z
I_{\Om}\big)^{-1}\big)^*\Big]^* \,\Big), \lb{4.33} \\
& \quad = \det\big(M_{0,\Om}^{N}(z)^{-1} M_{\Om}^{N}(z)\big).   \lb{4.34}
\end{align}
\end{remark}

\begin{remark}  \lb{r4.4}
$(i)$ For simplicity we focused on trace class nonlocal interactions 
$V\in \cB_1 \big(L^2(\Om; d^n x)\big)$ and Fredholm determinants only. 
Following our use of modified Fredholm determinants $\det_{p}(\cdot)$, 
$p\in\bbN$, in \cite{GMZ07}, one can develop 
all the results presented in this paper under the hypothesis 
$V\in \cB_k \big(L^2(\Om; d^n x)\big)$ for some $k \in \bbN$.  \\
$(ii)$ We closely followed \cite{GMZ07} and used the formalism based on the 
factorization of $V$ into $v u$ and symmetrized resolvent equations, etc. It is 
possible to avoid this factorization replacing the basic operator 
$u\big(H_{0,\Om}^{D,N}-zI_{\Om}\big)^{-1}v$ by 
$V\big(H_{0,\Om}^{D,N}-zI_{\Om}\big)^{-1}$   
(resp., by $\big(H_{0,\Om}^{D,N}-zI_{\Om}\big)^{-1}V$), etc. This applies, in 
particular, to the left-hand sides of \eqref{4.3}, \eqref{4.30}, and \eqref{4.33}.  
Of course, the latter observation also directly follows from identity \eqref{4.0}. \\
$(iii)$ Once more we emphasize that it is possible to remove the boundedness 
assumption on $\Om$ in Hypothesis \ref{h4.0} and assume that $\dOm$ is 
compact instead. 
\end{remark}


\end{document}